%% file: main.tex
\documentclass[reqno]{amsart}

\input{include/preambule.tex}

\begin{document}

\title[]{On chain rule and renormalization}

\author[R.V. Dribas]{Roman V. Dribas}
\email[R. D.]{dribas.rv@phystech.su}

\author[N.A. Gusev]{Nikolay A. Gusev}
\email[N. G.]{ngusev@phystech.su, n.a.gusev@gmail.com}

\address{Moscow Institute of Physics and Technology,
  9 Institutskiy per., Dolgoprudny, Moscow Region, 141700
}

\begin{abstract}
We discuss the relationship between the chain rule for the divergence operator and the renormalization property for weak solutions of the continuity equation.
We construct an example of bounded divergence-free vector field on the plane, which demonstrates that in general the first property is not sufficient for the second one.
\end{abstract}

\maketitle

\begin{center}
\end{center}

\tableofcontents

\section{Introduction}

Renormalization property was introduced by R.J. DiPerna and P.-L. Lions \cite{DiPernaLions1989} in connection with the problem of uniqueness of weak solutions of the continuity (or transport) equation with non-smooth vector field (in the context of hyperbolic conservation laws a very close concept was introduced earlier by S.N.~Kruzkov \cite{Kruzkov_1970}).
Sufficient conditions, under which the renormalization property holds, were obtained in \cite{DiPernaLions1989,Ambrosio2004,ABC14,BianchiniBonicatto2019} and many other papers.
(Note that uniqueness of weak solutions can also be studied by methods which do not involve the renormalization property, see e.g. \cite{CLR03,ABC14,Panov2008,Panov_2015,Panov2018,Panov_2025}.)
Using the same techniques, under certain assumptions on the vector field $\vi$ and on the scalar field $\rho \colon \R^d \to \R$
for any $\beta \in C^1(\R)$ with $\beta(0)=0$ and $\beta'\in L^\infty(\R)$ one can characterize the distribution $\div (\beta(\rho) \vi)$ in terms of the distributions $\div (\rho \vi)$ and $\div \vi$. For instance, in the smooth setting we have
\begin{equation}\label{chain-rule}
	\div(\beta(\rho) \vi) = (\beta(\rho) - \rho \beta'(\rho)) \div \vi + \beta'(\rho) \div (\rho \vi).
\end{equation}
In the non-smooth setting such equality may fail, even if all terms are well-defined, see e.g. 
\cite{CGSW2017,ModenaBuck2023}.
Similarly, in the non-smooth setting the renormalization property may fail as well (since it implies uniqueness of weak solutions to the corresponding Cauchy problem, and the latter may fail if $\vi$ is not sufficiently regular),
see e.g. \cite{Aizenman_1978,Depauw2003,CLR03,ModenaSzekelyhidi2019}.
However, in \cite{DiPernaLions1989} (see also \cite{ADLM07}) it was proved that if $\vi \in W^{1,p}_\loc(\R^d)$ (with $p \in [1,\infty]$), then \eqref{chain-rule} holds (in the sense of distributions) for all $\rho \in L^q_\loc(\R^d)$ (where $\frac{1}{p}+\frac{1}{q} = 1$) such that $\div(\rho \vi) \in L^1_\loc(\R^d)$.
The proof of this result is based on the same ideas, as the proof of the renormalization property (under similar assumptions) in \cite{DiPernaLions1989}.
This observation naturally leads to the following question: can the renormalization property be deduced directly from some identity similar to \eqref{chain-rule}, without explicitly using some form of weak differentiability of $\vi$?
In the present work we construct an example of a vector field $\vi$ which gives a negative answer to this question, even if $\div \vi = 0$.

In \cite{BG16} the chain rule problem for the divergence operator was studied in the two-dimensional setting under very generic assumptions on $\vi$ and $\rho$, and the distribution $\div(\beta(\rho) \vi)$ was characterized in terms of $\beta$, $\rho$ and the distributions $\div(\vi)$ and $\div(\rho \vi)$.
In particular, suppose for simplicity that $\vi\colon \R^2 \to \R^2$ is a bounded measurable divergence-free vector field with compact support. Let $\rho \in L^\infty(\R^2)$ and suppose that the distribution $\div(\rho \vi)$ can be represented by some signed bounded Borel measure $\mu$.
Then by \cite[Theorem 7.1]{BG16} there exist Borel functions $\rho^\pm \colon \R^2 \to \R$, which agree with $\rho$ a.e. on $\{\vi \ne 0\}$, such that for any $\beta \in C^1(\R)$
\begin{equation}
	\div (\beta(\rho) \vi) = f_1 \mu,
\end{equation}
where
\begin{equation*}
	f_1(x) = \begin{cases}
		\beta'(\rho^+(x)), & \rho^+(x) = \rho^-(x); \\
		\frac{\beta(\rho^+(x)) - \beta(\rho^-(x))}{\rho^+(x) - \rho^-(x)}, & \rho^+(x) \ne \rho^-(x).
	\end{cases}
\end{equation*}
Thus in the non-smooth setting, the generic chain rule formula does not have the same form as \eqref{chain-rule}, and the density of $\div (\beta(\rho) \vi)$ with respect to $\div(\rho \vi)$ is not given simply by $\beta'(\rho)$. 
So in this work our goal is to construct the vector field $\vi$ in such a way that $f_1$ does not involve the ``jump term'' $\frac{\beta(\rho^+(x)) - \beta(\rho^-(x))}{\rho^+(x) - \rho^-(x)}$, but at the same time $\vi$ does not have the renormalization property.

\subsection{Renormalization property}

Let $\vi\colon \R^d \to \R^d$ be a bounded measurable vector field with compact support. We will always assume that 
\begin{equation}\label{div-free}
	\div \vi = 0
\end{equation}
(in the sense of distributions).
Let $T > 0$, 
$\rho_0 \in L^\qq(\R^d)$ and $g \in L^\qq(\R^d)$. 
Consider the Cauchy problem for continuity equation
\begin{equation}\label{Cauchy-propblem}
	\begin{cases}
		\d_t \rho + \div(\rho \vi) = g, \\
		\rho|_{t=0} = \rho_0,
	\end{cases}
\end{equation}
where the solution $\rho$ belongs to $L^\infty([0,T]; L^\qq(\R^d))$.
The problem \eqref{Cauchy-propblem} is understood in the sense of distributions, i.e.
for any $\fhi\in C^1_c([0,T)\times \R^d)$
\begin{equation*}
	\int_0^T \int_{\R^d} (\rho \d_t \fhi + \rho \vi \nabla \fhi)\, dx \, dt
	+ \int_{\R^d} \rho_0 \fhi(0, \cdot) \, dx + \int_0^T \int_{\R^d} g \fhi \, dx \,dt = 0.
\end{equation*}

\begin{defn}
    A function $\beta \in C^1(\R)$ is \emph{admissible}, if $\beta' \in L^\infty(\R)$ and $\beta(0) = 0$.
\end{defn}

\begin{defn}
    A solution $\rho \in L^\infty([0,T]; L^\qq(\R^d))$ of \eqref{Cauchy-propblem} is \emph{renormalized}, if for all admissible functions $\beta \in C^1(\R)$ the function $\beta(\rho)$ solves
    \begin{equation}
        \begin{cases}
            \d_t\beta(\rho) + \div(\beta(\rho) \vi) = \beta'(\rho) g, \\
            \beta(\rho)|_{t=0} = \beta(\rho_0).
        \end{cases}
    \end{equation}
\end{defn}

\begin{defn}
	If for all $\rho_0 \in L^\qq(\R^d)$ and all $g\in L^\qq(\R^d)$ (resp. $g=0$) all solutions $\rho \in L^\infty([0,T]; L^\qq(\R^d))$ of \eqref{Cauchy-propblem} are renormalized, then we say that $\vi$ has the \emph{non-homogeneous (resp. homogeneous) renormalization property}.
\end{defn}

The standard term \emph{renormalization property} usually refers to the homogeneous renormalization property defined above.

When $d=2$, by \eqref{div-free} there exists (unique up to an additive constant) Lipschitz function $f\colon \R^2\to \R$ such that
\begin{equation*}
	\vi = \nabla^\perp f,
\end{equation*}
where $\nabla^\perp f \equiv (-\d_2 f, \d_1 f)$.
The~function $f$ is known as the \emph{stream function} (or \emph{Hamiltonian}) of $\vi$.
Given a Borel map $f\colon \R^2 \to \R$ and a Borel measure $\mu$ on $\R^2$, let $f_\# \mu$
denote the image of $\mu$ under $f$, i.e. for any Borel $B\subset \R$ we have $(f_\# \mu)(B) = \mu(f^{-1}(B))$.
Recall that, given two Borel measures $\mu$ and $\nu$ on $\R$, we write $\mu \perp \nu$ if and only if $\mu$ and $\nu$ are mutually singular, i.e. there exist Borel sets $A,B\subset \R$ such that $A\cap B = \emptyset$, $\R = A \cup B$ and $\mu(A) = \nu(B) = 0$.

\begin{defn}\label{WSP}
    The function $f$ has \emph{weak Sard property}, if
    \begin{equation}
        f_\#(\1_{S \cap E^*}\L^2) \perp \L^1,
    \end{equation}
    where $S = \{\nabla f = 0\}$ is the critical set of $f$, 
    \begin{equation}
    	E^* = \bigcup_{y\in \R} E_y^*
    	\label{E-star}
    \end{equation}
    and $E_y^*$ denotes the union of all connected components $C$ of the level set $\{f=y\}$ such that $\H^1(C) > 0$.
\end{defn}

Here $\H^s$ is the $s$-dimensional Hausdorff measure, and $\mathscr{L}^d$ denotes the $d$-dimensional Lebesgue measure.
The set $E^*$ is Borel, see \cite[Proposition 6.1]{ABC13}.

The following elegant necessary and sufficient condition for the homogeneous renormalization property (and uniqueness of weak solutions) was obtained in \cite{ABC14} (see also \cite[Theorem 1.5]{GK25} for an extension to square-integrable vector fields).

\begin{thm}\label{thm-ABCGK}(\cite[Theorem 4.7]{ABC13})
    Let $d=2$ and 
    let $\vi \in L^\infty(\R^2)$ be vector field with compact support. 
    Let $f$ denote a stream function of $\vi$.
    Then the following statements are equivalent:
    \begin{enumerate}
        \item $f$ has weak Sard property,
        \item $\vi$ has the homogeneous renormalization property,
        \item for any initial condition $\rho_0\in L^\qq(\R^d)$ Cauchy problem \eqref{Cauchy-propblem} with the field $\vi$ has at most one solution in class $L^\infty(0,T;L^\qq(\R^d))$.
    \end{enumerate}
\end{thm}

\subsection{Chain rule property}
Now we turn to the steady version of the renormalization property, which for convenience will be called the chain rule property.

\begin{defn}\label{dfn-chain-rule}
	Suppose that $\vi\colon \R^d \to \R^d$ be a bounded measurable vector field with $\div \vi =0$ (in the sense of distributions).
	\begin{enumerate}[label=\emph{(\roman*)}]
	\item If for all $\rho \in L^\qq(\R^d)$, and for all admissible $\beta \in C^1(\R)$
	\begin{equation}
		\div(\rho \vi) = 0 \qquad\text{implies}\qquad \div(\beta(\rho) \vi) = 0,
	\end{equation}
	then we say that $\vi$ has the \emph{homogeneous chain rule property}.

	\item If for all $\rho \in L^\qq(\R^d)$, and for all $g \in L^\qq(\R^d)$ there exists a measurable function $\widehat{\rho}$ which agrees with $\rho$ a.e. on the set $\{\vi \ne 0\}$ such that for all admissible $\beta \in C^1(\R)$
    \begin{equation}
        \div(\rho \vi) = g \qquad\text{implies}\qquad \div(\beta(\rho) \vi) = \beta'(\widehat{\rho}) g,
    \end{equation}
    then we say that $\vi$ has the \emph{non-homogeneous chain rule property}.
	\end{enumerate}
\end{defn}

Note that $\widehat{\rho}$ in general cannot be replaced with $\rho$, since $\div(\rho \vi)$ does not depend on the values of $\rho$ on the set $\{\vi = 0\}$, but $\beta'(\rho) g$ depends on such values (in general).
The dependence of $\div(\rho \vi)$ on the values of $\rho$ on the set $\{\vi = 0\}$ is related to the non-locality of the divergence operator, see e.g. \cite{ABC13}.

\subsection{Main result}
For $d=2$ the homogeneous chain rule property holds for any bounded measurable compactly supported divergence-free vector field $\vi\colon \R^2 \to \R^2$ by \cite[Theorem 7.1]{BG16}, see also \cite[Theorem 1.4]{GK25} for a generalization of this result.
But the homogeneous renormalization property is equivalent to the weak Sard property,
and in \cite[Section 4]{ABC13} there was constructed an example of Lipschitz function which does not have the weak Sard property.
Hence the homogeneous chain rule property is always strictly weaker than the homogeneous renormalization property.

In this work we prove that even non-homogeneous chain rule property is not sufficient for the homogeneous renormalization property:

\begin{thm}\label{main}
    Let $d = 2$. Then there exists a compactly supported vector field $\vi\colon \R^2 \to \R^2$
    such that $\vi$ has the non-homogeneous chain rule property, but does not have the homogeneous renormalization property.
\end{thm}

In order to obtain the desired vector field $\vi$ in Section~\ref{sect-construction} we construct its stream function $f$ in such a way that $f$ does not have the weak Sard property. By Theorem~\ref{thm-ABCGK} it follows that $\vi$ does not have the homogeneous renormalization property.

Our construction of $f$ is similar to the one introduced in \cite[Section 4]{ABC14}, but in our modification we ensure that the disintegration of the Lebesgue measure (restricted to the critical set of $f$) under $f$ is such that the non-homogeneous chain rule property holds, unlike in the original example from \cite{ABC14}.
The proof of the non-homogeneous chain rule property is given in Section~\ref{sect-proof}, see also Lemma~\ref{atomless-chain-rule}.

\section{Preliminaries}

\subsection{Disintegration of measures}\label{subsect-disintegration}
Let us recall that a family of Borel measures $\{\mu_y\}_{y\in \R}$ on $\R^d$ is Borel
if for any Borel set $A\subset \R^d$ the function $y\mapsto \mu_y(A)$ is Borel.
Given two Borel measures $\mu$ and $\nu$ on $\R$, as usual we write $\mu \ll \nu$ if and only if $\mu$ is absolutely continuous with respect to $\nu$, i.e. for any Borel set $N\subset \R$ with $\nu(N) = 0$ we have $\mu(N) = 0$.
The following decomposition result holds, see e.g. \cite[volume II, \S 10.4]{bogachev}:

\begin{thm}[disintegration theorem]
	Let $\mu$ be a finite non-negative Borel measure on $\R^d$
	and let $f\colon \R^d \to \R$ be a Borel function.
	Let $\nu$ be a Borel measure on $\R$ such that $f_\# \mu \ll \nu$.
	Then there exists a Borel family $\{\mu_y\}_{y\in \R}$
	such that
	$\mu_y$ is concentrated on $\{f=y\}$ for $\nu$-a.e. $y$ and
	\begin{equation*}
		\mu = \int_{\R} \mu_y \, d\nu(y)
	\end{equation*}
	in the sense that
	\begin{equation}\label{mu-disint}
		\mu(A) = \int_{\R} \mu_y(A)\, d\nu(y)
	\end{equation}
	for all Borel sets $A\subset \R^d$.
	If $\{\tilde \mu_y\}_{y\in \R}$ is another such family, then $\tilde \mu_y = \mu_y$ for $\nu$-a.e. $y$.
\end{thm}
The family $\{\mu_y\}$ given by the theorem above is called \emph{the disintegration of $\mu$ with respect to $f$ (and $\nu$)}. 
Whenever $f_\# \mu \ll \L^1$ it is convenient to assume that $\nu = \L^1$.

For any non-negative Borel function $\fhi \colon \R^d \to \R$ the following generalization of \eqref{mu-disint} holds:
\begin{equation}
	\int_{\R^d} \fhi \, d\mu = \int_{\R} \int_{\R^d} \fhi \, d\mu_y \, d\nu(y).
\end{equation}
(Indeed, when $\fhi$ is an indicator of a Borel set, the equality above follows from \eqref{mu-disint}.
Then by standard arguments it is extended for all bounded non-negative Borel functions $\fhi$ and then for all non-negative Borel functions.)

If $\mu$ is finite measure, then by \eqref{mu-disint} with $A=\R^d$ the map $y\mapsto \mu_y(\R^d)$ belongs to $L^1(\nu)$.

Using Hahn decomposition it is straightforward to extend the disintegration theorem for finite signed measures $\mu$
(in this case one has to replace $f_\# \mu \ll \nu$ with $f_\# |\mu| \ll \nu$).

\subsection{Coarea formula}\label{subsect-coarea}

Let us briefly recall the coarea formula for Lipschitz functions,
which can be viewed as a particular case of the disintegration theorem when $\mu = \1_{\{\nabla f \ne 0\}}\L^2$ and $f \colon \R^2 \to \R$ is Lipschitz.
See for instance \cite[Theorem 3.40]{AFP} and \cite[Theorem 1.1]{MSZ_2003}.

\begin{thm}[coarea formula]
	Suppose that $f\colon \R^2 \to \R$ is a compactly supported Lipschitz function.
	Then
	\begin{equation}
		\1_{\{\nabla f \ne 0\}}\L^2 = \int_{\R} \frac{1}{|\nabla f|} \1_{\{f=y\}} \H^1 \, dy.
	\end{equation}
\end{thm}
Here $|\cdot|$ denotes the standard Euclidean norm on $\R^2$.

\subsection{Chain rule for functions of bounded variation}
As usual for any open interval $\Omega \subset \R$ we denote with
$BV(\Omega)$ the set of all integrable functions $u\colon \Omega \to \R$
whose distributional derivative is a finite Borel measure $Du$ on $\Omega$.
We will need the Vol'pert chain rule formula for a composition of smooth function with a function of bounded variation,
see e.g. \cite[Theorem 3.99]{AFP}. 

Recall that given $u \in BV(\Omega)$ the \emph{approximate discontinuity set $S_u \subset \Omega$} is the complement
of the set of all $x\in \Omega$ for which there exists $u^*(x)$ such that
\begin{equation*}
	\lim_{r\to 0} \frac{1}{2r}\int_{x-r}^{x+r} |u(y) - u^*(x)| \, dy = 0.
\end{equation*}
The function $u^* \colon \Omega \setminus S_u \to \R$ is called the \emph{precise representative of $u$}.
It is well-known that if $u \in BV(\Omega)$, then the set $S_u$ is at most countable (see e.g. \cite[Theorem 3.78]{AFP} with $N=1$),
and that $S_u$ coincides with the set of all atoms of $Du$.

For our purposes the following simplified result will be sufficient:

\begin{thm}\label{thm-AFP-3-99}(chain rule for BV functions with atomless derivative)\label{chain-rule-BV}
	Let $\Omega \subset \R$ be a bounded open interval and suppose that $u \in BV(\Omega)$ has atomless $Du$.
	Then for any $\beta\in C^1(\R)$ we have $\beta(u) \in BV(\Omega)$ with
    \begin{equation}
        D \beta(u) = \beta'(u^*) Du,
    \end{equation}
    where $u^*$ is the precise representative of $u$.
\end{thm}

\goodbreak

\subsection{Disintegration of the divergence equation}

We say that a compactly supported Lipschitz function $f\colon \R^2 \to \R$ is \emph{monotone} if for all $y\in \R \setminus \{0\}$
the level set $\{f=y\}$ is connected. (Almost identical definition of monotonicity is given for instance in \cite[p. 358]{Engelking1989}, with the additional requirement that $\{f=0\}$ is connected. See also \cite[Theorem~5.18]{GK25}.)
By the well-known results in Geometric Measure Theory (see e.g. \cite[Theorem~2.5 and Lemma~2.11 (ii)]{ABC14}),
for a.e. $y\in \R$ we have $\H^1(\{f=y\}) < \infty$, and for a.e. $y\in \R$ such that $\H^1(\{f=y\})>0$ there exist $L_y>0$ and a closed simple Lipschitz curve $\gamma \colon [0,L_y] \to \R^2$
such that $\{f=y\}$ coincides with $\gamma([0,L_y])$ (up to $\H^1$-negligible subsets) and
\begin{equation*}
	\gamma' = (\nablap f) \circ \gamma
\end{equation*}
a.e. on $[0,L_y]$. Such $\gamma$ will be called \emph{consistent parametrization of $\{f=y\}$}.
Evidently $\gamma$ depends also on $y$, but we omit this dependence in order to simplify the notation.
For convenience we will view $\gamma$ as an $L_y$-periodic function on $\R$.

\begin{prop}\label{disint-dive}
	Let $f\colon \R^2 \to \R$ be a monotone compactly supported Lipschitz function and $\vi = \nabla^\perp f$.
	Suppose that $\mu = \1_{\{\nabla f = 0\} \cap E^*} \L^2$ and $f_\# \mu \ll \L^1$.
	Let $\{\mu_y\}_{y\in \R}$ denote the disintegration of $\mu$ with respect to $f$ (and $\L^1$).
	Then Borel functions $\rho \in L^\qq(\R^2)$ and $g\in L^\qq(\R^2)$ solve (in the sense of distributions)
	\begin{equation}\label{dive-rho-v-eq-g}
		\div (\rho \vi) = g
	\end{equation}
	if and only if $g=0$ a.e. on $\R^2 \setminus E^*$ and for a.e. $y$ such that $\H^1(\{f=y\})>0$
	we have $\widetilde{\rho} \in BV_\loc(\R)$ and
	\begin{equation}\label{dive-along-curve-param}
		D \widetilde{\rho} = \widetilde{g} \cdot (\L^1 + \widetilde{\mu}_y)
	\end{equation}
	in $\mathscr{D}'(\R)$, where $\widetilde{\rho} = \rho \circ \gamma$, $\widetilde{g} = g \circ \gamma$,
	$\widetilde{\mu}_y$ is the image of $\mu_y$ under $\gamma^{-1}$, and $\gamma$ is a consistent parametrization of $\{f=y\}$.
	(The set $E^*$ was defined in \eqref{E-star}.)
\end{prop}

The proof of this result is almost identical to the proof of \cite[Lemma 3.9 and Lemma 4.4]{ABC14}, although it can be considerably simplified.
See also \cite[Proposition 7.2]{GK25} for the particular case when $g=0$ a.e. on $\{\nabla f = 0\}$. 

\begin{lem}\label{atomless-chain-rule}
	Under the assumptions of Proposition~\ref{disint-dive}, if $\mu_y$ has no atoms for a.e. $y$, then
	$\vi$ has the chain rule property (in the sense of Definition~\ref{dfn-chain-rule}).
\end{lem}

\begin{proof}
	By Proposition~\ref{disint-dive} for a.e. $y$ such that $\H^1(\{f=y\})>0$ we have \eqref{dive-along-curve-param}.
	Since the parametrization $\gamma \colon [0,L_y] \to \R^2$ of $\{f=y\}$ is injective on $[0,L_y)$,
	the measure $\widetilde{\mu}_y$ is atomless by the assumptions,
	hence the distribution $D\widetilde{\rho}$ is an atomless measure.
	Then there exists a continuous function $\overline{\rho} \colon [0, L_y] \to \R$ such that $\overline{\rho} = \widetilde{\rho}$ a.e.
	Then by Theorem~\ref{chain-rule-BV} for any admissible $\beta$ we have
	\begin{equation}
		D\beta(\overline{\rho}) = \beta'(\overline{\rho})\widetilde{g}\cdot(\L^1 + \widetilde{\mu}_y).
	\end{equation}
	The above equation holds for a.e. $y$ such that $\H^1(\{f=y\})>0$.
	
	It is known (see e.g. \cite[Claim A.1]{BG16}) that there exists a measurable function $\widehat{\rho} \colon \R^2 \to \R$ such that $\rho = \widehat{\rho}$ a.e. on $\{\vi \ne 0\}$ and for a.e. $y\in \R$ such that $\H^1(\{f=y\})>0$ and $\{f=y\}$ has a Lipschitz parametrization $\gamma\colon [0, L_y] \to \R^2$ we have
	\begin{equation*}
		\widehat{\rho}(\gamma(t)) = \overline{\rho}(\gamma(t)).
	\end{equation*}
	
	Now, using Proposition~\ref{disint-dive} again (but this time in the other direction) we conclude that
	\begin{equation*}
		\div (\beta(\widehat{\rho}) \vi) = \beta'(\widehat{\rho}) g.
		\qedhere
	\end{equation*}
\end{proof}

\section{Construction of the stream function}\label{sect-construction}

\subsection{Construction parameters}

Let $(a_n)_{n \in \N}, (b_n)_{n \in \N}$ be decreasing sequences of positive real numbers such that
\begin{equation}\label{an-bn}
    a_n \sim b_n \sim \frac{1}{n^2 4^n}, \quad n\to\infty.
\end{equation}
Then the following quantities
\begin{equation}
    \hat{a} := \sum_{n=0}^\infty 2^{2n+3} a_n, \quad \hat{b} = \sum_{n=0}^\infty 2^{2n+2} b_n
\end{equation}
are finite. Let $\delta > 0$, define
\begin{equation}
    c_0 := \delta + \hat{a}, \quad d_0 := \delta + \hat{b}.
\end{equation}

\subsection{Rectangles \texorpdfstring{$C_n$}{Cn}}

\begin{figure}[h!]
    \centering
    \includegraphics[width=0.9\linewidth]{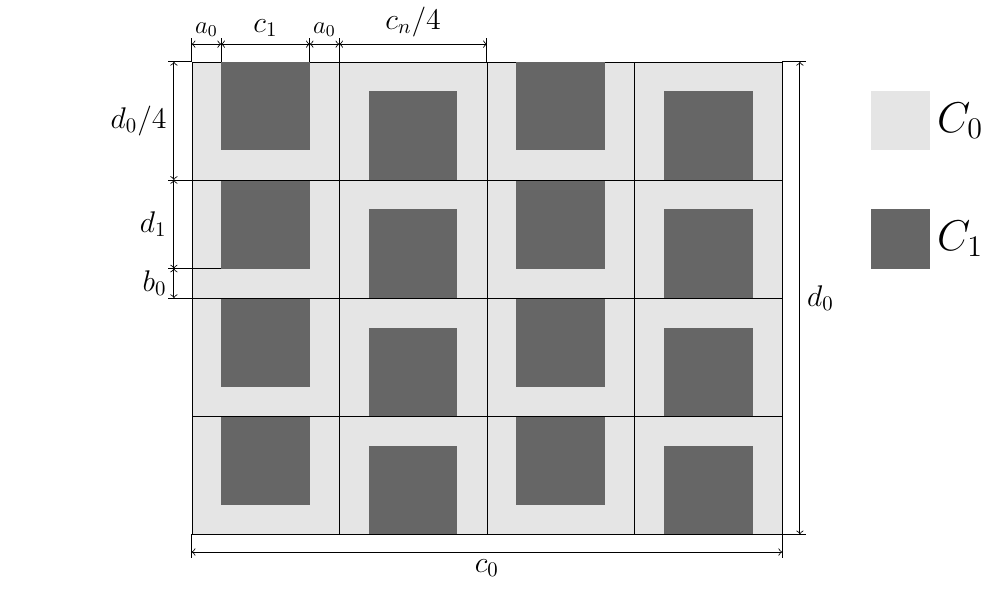}
    \caption{The sets $C_n$ for $n = 0, 1$}
    \label{pic-set}
\end{figure}

Let $C_0$ be closed rectangle with width $c_0$ and height $d_0$. Define $C_1$ to be the union of 16 closed rectangles with sizes
\begin{equation}
    c_1 := \frac{1}{4}c_0 - 2 a_0, \quad d_1 := \frac{1}{4}d_0 - b_0
\end{equation}
like in Figure \ref{pic-set}. 

Then we iterate this construction: if $C_n$ is the union of $16^n$ pairwise disjoint closed rectangles with width $c_n$ and height $d_n$, then $C_{n+1}$ is the union of $16^{n+1}$ disjoint closed rectangles with sizes
\begin{equation}
    c_{n+1} := \frac{1}{4} c_n - 2 a_n, \quad d_{n+1} := \frac{1}{4}d_n - b_n.
\end{equation}
It is easy to see that
\begin{equation}
    4^n c_n = c_0 - \sum_{m=0}^{n-1} 2^{2m+3} a_m \searrow \delta \quad\text{and}\quad 4^n d_n = d_0 - \sum_{m=0}^{n-1} 2^{2m+2} b_m \searrow \delta
\end{equation}
as $n\to \infty$, so $c_n, d_n$ are always strictly positive and satisfy
\begin{equation}\label{cn-dn}
    c_n \sim d_n \sim \frac{\delta}{4^n}.
\end{equation}

Denote $C = \bigcap_{n=1}^\infty C_n$. We have
\begin{equation}
    \L^2(C) = \lim_{n \to \infty} \L^2(C_n) = \lim_{n \to \infty} 16^n c_n d_n = \delta^2 >0.
\end{equation}

\subsection{Construction of the stream function}

Let us construct a sequence of Lipschitz and piecewise smooth functions $f_n: \R^2 \to \R$ as follows.

\begin{figure}
\centering
\begin{subfigure}{0.49\textwidth}
    \includegraphics[width=\textwidth]{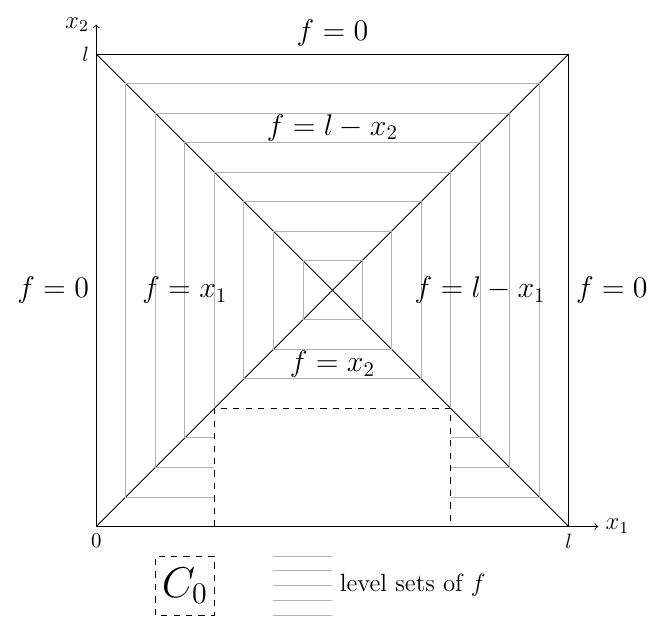}
    \caption{Level sets of $f_0$.}
    \label{pic1a}
\end{subfigure}
\begin{subfigure}{0.49\textwidth}
    \includegraphics[width=\textwidth]{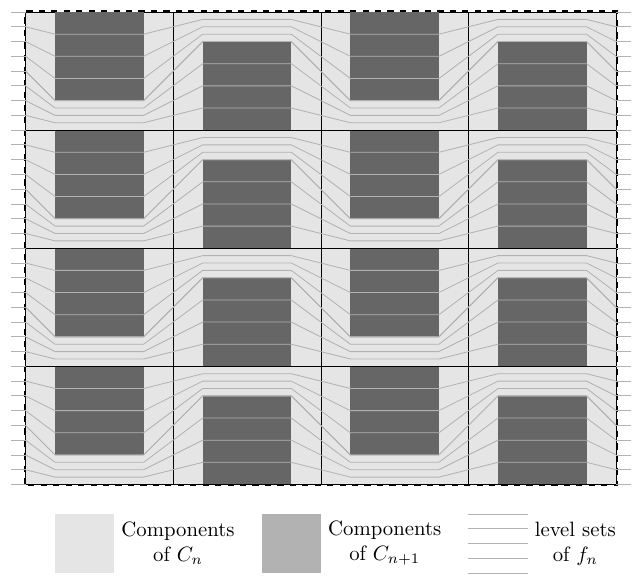}
    \caption{Level sets of $f_n$ in each component of $C_n$.}
    \label{pic1b}
\end{subfigure}
\caption{Level sets of the approximate stream functions $f_n$.}
\label{pic1}
\end{figure}

Denote $l := c_0 + 2d_0$.
Let the function $f_0$ be zero in $\R^2 \setminus [0,l]^2$.
For $(x_1,x_2) \in [0,l]^2 \setminus C_0$ the function $f_0$ is defined by the following formulae (see Figure \ref{pic1a}):
\begin{align*}
    f_0(x_1,x_2) = x_2,& \quad x_2 < x_1 \text{ and } x_2 + x_1 \leq l, \\
    f_0(x_1,x_2) = l - x_1,& \quad x_2 < x_1 \text{ and } x_2 + x_1 > l, \\
    f_0(x_1,x_2) = x_1,& \quad x_1 < x_2 \text{ and } x_1 + x_2 \leq l, \\
    f_0(x_1,x_2) = l - x_2,& \quad x_1 < x_2 \text{ and } x_2 + x_1 > l.
\end{align*}
In $C_0$ the function $f_0$ is defined by its level sets, described in Figure \ref{pic1b}.

The function $f_1$ agrees with $f_0$ in $\R^2 \setminus C_1$, while in each component of $C_1$ it is defined by its level curves, described in Figure \ref{pic1b}. And so on for $n = 2,3,\dots$.

Denote $h_n := f_n - f_{n-1}$ for all $n > 0$.

\begin{lem}
    We have the following estimates:
    \begin{equation}\label{norm-hn-grad-hn}
        \|h_n\|_\infty = O(n^4 8^{-n}), \quad \|\nabla h_n\|_\infty = O(n^4 2^{-n})
    \end{equation}
\end{lem}

\begin{proof}

\begin{figure}
\centering
\begin{subfigure}[c]{0.49\textwidth}
    \includegraphics[width=1\textwidth]{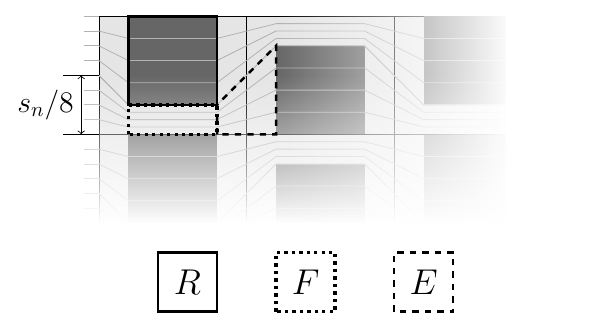}
    \centering
    \caption{Sets $F$ and $E$.}
    \label{pic2a}
\end{subfigure}
\begin{subfigure}[c]{0.49\textwidth}
    \includegraphics[width=1\textwidth]{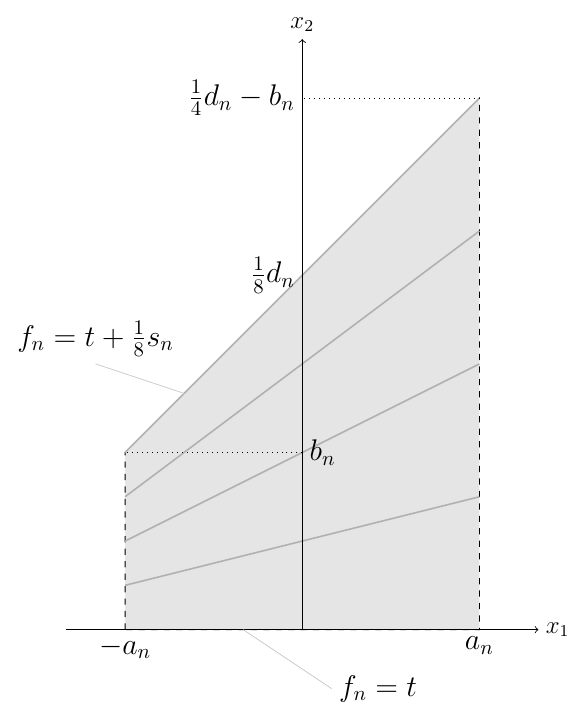}
    \caption{Level sets of $f_n$ in the set $E$.}
    \label{pic2b}
\end{subfigure}
\caption{The level sets of $f_n$ in each connected component of $C_n$.}
\label{pic2}
\end{figure}
    
Let us write the explicit formulae of $f_n$.

Denote oscillation of a function $f$ on a set $E$ as $\osc(f, E) := \sup\limits_{x \in E} f(x) - \inf\limits_{x \in E} f(x)$.

For any connected component $Q$ of $C_n$ let $s_n := \osc(f_n,Q)$ (by symmetry $s_n$ does not depend on the choice of $Q$).
It is easy to see from the Figure \ref{pic1b} that
\begin{equation}\label{sn}
s_{n+1} = \frac{s_n}{8}, \quad s_n = 8^{-n} s_0 = 8^{-n} d_0.
\end{equation}

Let us estimate $L^\infty$-norm of $\nabla f_n$ on $C_n$.
In order to do this, we will estimate $\nabla f_n$ separately on any connected component $R$ of $C_{n+1}$,
as well as on the sets $E$ and $F$ introduced on Figure~\ref{pic2a}.

If $R$ is a connected component of $C_{n+1}$ then for all $x\in R$ the function $f_n(x)$ does not depend on $x_1$ and is a linear function of $x_2$. So $|\nabla f| = |(0, \frac{\osc(f_n,R)}{\operatorname{height}(R)})| = \frac{s_n}{8 d_{n+1}}$ in $R$.

It is easy to see from Figure \ref{pic2a} that $|\nabla f_n| = \frac{s_n}{8b_n}$ in $F$.

If we choose the axes as in Figure \ref{pic2b} and denote by $t$ the value of $f_n$ at the bottom of $E$, then we readily check that for every $x \in E$, $f_n(x)$ is given by
\begin{equation}
    f_n(x) = t + \frac{a_n s_n x_2}{a_n d_n + (d_n - 8 b_n) x_1}
\end{equation}
and direct calculations yields
\begin{equation}\label{norm-fn}
    \nabla f_n = \frac{(-(d_n - 8 b_n) (f_n(x) - t), a_n s_n)}{a_n d_n + (d_n - 8 b_n) x_1}.
\end{equation}
Due to \eqref{an-bn} and \eqref{cn-dn} we have $d_n - 8 b_n \sim \delta \cdot 4^{-n}$. In particular, $d_n - 8 b_n > 0$ for $n$ sufficiently large and due to $x_1 \geq -a_n$ we have
\begin{equation}
    a_n d_n + (d_n - 8 b_n) x_1 \geq 8 a_n b_n.
\end{equation}
Since $|f_n - t| \leq s_n$ we have $|-(d_n - 8 b_n) (f_n(x) - t)| \leq (d_n - 8 b_n) s_n$. Then
\begin{equation}\label{norm-grad-fn}
    |\nabla f_n| \leq \frac{(d_n - 8 b_n) s_n + a_n s_n}{8 a_n b_n} = \frac{s_n}{8 b_n}\left(\frac{d_n - 8 b_n}{a_n} + 1\right) \text{ \ in \ } E.
\end{equation}

Combining the above estimates, we conclude that
\begin{equation*}
	\|\nabla f_n\|_{L^\infty(C_n)} = O(n^4 2^{-n}), \ \ n\to \infty.
\end{equation*}
By construction $f_n = f_{n-1}$ in $\R^2 \setminus C_n$ therefore $\supp h_n \subset C_n$ and
\begin{equation}
    \|\nabla h_n\|_\infty \leq \|\nabla f_n \|_{L^\infty(C_n)} + \|\nabla f_{n-1}\|_{L^\infty(C_{n})} \leq \|\nabla f_n \|_{L^\infty(C_n)} + \|\nabla f_{n-1}\|_{L^\infty(C_{n-1})} = O(n^4 2^{-n}),
\end{equation}
and since distance of a point in $C_n$ from $\R^2 \setminus C_n$ is of order $c_n = O(4^{-n})$ by the Mean Value Theorem we have
\begin{equation}
    \|h_n\|_\infty = O(n^4 8^{-n}).
\end{equation}
This completes the proof of the Lemma.
\end{proof}

For every $x \in \R^2$ define
\begin{equation}\label{defn-of-f}
    f(x) := \lim_{n \to \infty} f_n(x) = f_0 + \sum\limits_{n=1}^\infty h_n(x).
\end{equation}

Note that $f_n$ uniformly converges to $f$.

\subsection{Properties of the stream function}

Due to norm estimates we have $\sum_{n=1}^\infty \|h_n\|_{W^{1,\infty}} < \infty$, then $f$ is well-defined Lipschitz function on $\R^2$, and by construction it has compact support.

Let us show other properties of the function $f$.

\begin{thm}\label{thm-prop-of-f}
    We have
    \begin{enumerate}[label=\emph{(\roman*)}]
        \item $C$ is the critical set of $f$ in the square $[0,l]^2$ (up to $\L^2$-negligible subsets),
        \item $\L^1(f(C)) = d_0$,
        \item $f_\#(\1_C \cdot \L^2) = \frac{\delta^2}{d_0} \1_{f(C)} \cdot \L^1$. In particular $f$ does not have weak Sard property.
    \end{enumerate}
\end{thm}

\begin{proof}
    (i) For every $n \geq 0$ we have $f = f_n + \sum_{k=n+1}^\infty h_k$. Since $\nabla f_n$ and $\nabla h_k$ have pairwise disjoint supports, then estimates \eqref{norm-grad-fn} and \eqref{norm-hn-grad-hn} yield
    \begin{equation}
        \|\nabla f\|_{L^\infty(C_n)} = O(n^4 2^{-n}).
    \end{equation}
    So the Lipschitz constant of $f$ on each component of $C_n$ is of order $O(n^4 2^{-n})$. Since $C$ is contained in the interior of $C_n$, it follows that
    \begin{equation}
        \operatorname{lim sup}\limits_{y \to x} \frac{|f(y) - f(x)|}{|x - y|} = O(n^4 2^{-n})
    \end{equation}
    for every $x \in C$, and taking the limit as $n \to \infty$ we obtain $f$ is differentiable at $x$ and $\nabla f(x) = 0$.

    (ii) By construction we have $f(C) = f_n(C_n) = [0,d_0]$. Therefore
    \begin{equation}
        \L^1(f(C)) = d_0.
    \end{equation}

    (iii) For a function $g(x,y)=y$ we have $g_\#(\1_{[0,1]^2} \L^2) = \1_{[0,1]} \L^1$. Therefore by construction of the functions $f_n$ we have
    \begin{equation}
        (f_n)_\#(\1_C \L^2) = \frac{\delta^2}{d_0} \1_{f(C)} \L^1.
    \end{equation}
    Since $f_n$ uniformly converges to $f$, then $(f_n)_\#(\1_C \L^2)$ weakly-$*$ converges to $f_\#(\1_C \L^2)$, which yields that $f_\#(\1_C \L^2) = \frac{\delta^2}{d_0} \1_{f(C)} \L^1$.  
\end{proof}

\begin{prop}\label{prop-level-sets}
    For any $y\in(0,l/2)$ the level set $\{f=y\}$ is a closed simple Lipschitz curve.
    Furthermore, $E^* = (0,l)^2 \setminus \{(l/2,l/2)\}$.
\end{prop}

\begin{proof}
For any $y\in (0,l/2)$ the set $\{f=y\} \cap (\R^2 \setminus C_0)$ is simple piecewise linear curve.

For every $n$ the set $\{f_n=y\} \cap C_0$ is the graph of some piecewise linear function $\varphi_n(x_1)$.
By construction $f_n$ agrees with $f_{n+k}$ on $\R^2 \setminus C_n$ so $\|\varphi_n - \varphi_{n+k}\|_\infty \leq d_n$.
Since $d_n \to 0$ as $n \to \infty$, $\varphi_n$ uniformly converges to some continuous function $\varphi$ whose graph coincides with $\{f=y\} \cap C_0$.

Let us estimate $\H^1(\{f_n = y\} \cap C_0)$. Note that $\{f_n = y\} \cap C_0$ consists of $4^n$ horizontal pieces and $4^n$ inclined pieces (see Figure \ref{pic2}). Length of the horizontal ones could be estimated by $c_0$, and length of the inclined piece could be estimated by $\sqrt{(\frac{1}{4} d_n - 2 b_n)^2 + (2 a_n)^2} \sim \frac{\delta}{4^n}$. So there is some constant $M$ such that $\H^1(\{f_n = y\} \cap C_0) \leq M$ for all $n$. Since $\varphi_n$ uniformly converges to $\varphi$, then
\begin{equation}
    \H^1(\{f = y\} \cap C_0) \leq \liminf_{n \to \infty} \H^1(\{f_n = y\} \cap C_0).
\end{equation}
By e.g. \cite[Lemma 3.2]{Fal85} the curve $\{f = y\} \cap C_0$ has Lipschitz parametrization.

Therefore $\{f=y\}$ is concatenation of two simple Lipschitz curves. 
Thus for any $x\in f^{-1}((0,l/2))$ the corresponding level set $\{f=f(x)\}$ is a closed simple Lipschitz curve with strictly positive finite length. 
Hence $f^{-1}((0,l/2)) \subset E^*$.
It remains to note that $f^{-1}(l/2) = \{(l/2,l/2)\}$ has zero length, and $\H^1(\{f=0\}) = +\infty$.
\end{proof}

\goodbreak

\section{Chain rule property}\label{sect-proof}

Let $f$ denote the stream function constructed in the previous section.
Let $\mu = \1_{C} \L^2$.
By Theorem~\ref{thm-prop-of-f} we have $f_\# \mu \ll \L^1$.
Let $\{\mu_y\}_{y\in \R}$ denote the disintegration of $\mu$ with respect to $f$ (and $\L^1$).
By Proposition~\ref{prop-level-sets} the interior of $C_0$ is a subset of $E^*$, then $\{\nabla f = 0\} \cap E^* = C$.
Thus in order to show that $\vi := \nablap f$ has the non-homogeneous chain rule property, in view of Lemma~\ref{atomless-chain-rule} it remains to prove the following:

\begin{prop}\label{atomless}
	For a.e. $y$ the measure $\mu_y$ has no atoms.
\end{prop}

\begin{proof}
    \begin{figure}
        \centering
        \includegraphics[width=0.9\linewidth]{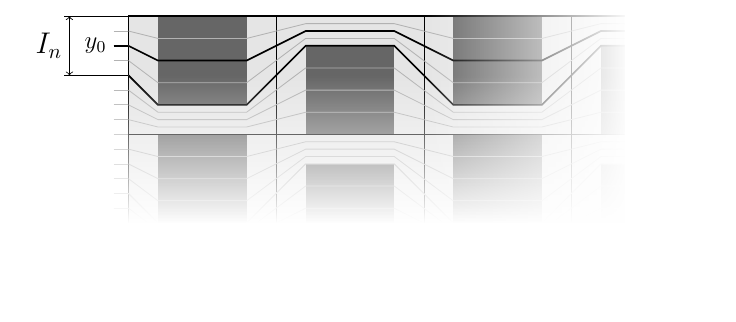}
        \caption{Level set as the intersection of nested curvilinear stripes.}
        \label{pic3}
    \end{figure}
    Let us show that for a.e. $y$ the measure $\mu_y$ is atomless.
    
    For any Borel set $A \subset C_0$ the map $y \mapsto \mu_y(A)$ belongs $L^1(\R)$ (see Section~\ref{subsect-disintegration}), therefore, a.e. $y\in [0,d_0]$ are Lebesgue points of this map.
    Let $\mathscr{R}_m$ denote the set of all connected components of $C_m$ and let $\mathscr{R} := \bigcup_{m\in \N} \mathscr{R}_m$.
    Since~$\mathscr{R}$ is countable, there exists a negligible set $N \subset [0, d_0]$ such that any $y_0 \in [0,d_0] \setminus N$ is a Lebesgue point of $y \mapsto \mu_y(A)$ \emph{for all} $A \in \mathscr{R}$.
    
    Fix $m \in \N$ and $A \in \mathscr{R}_{m}$.
    For every $n \in \N$ let $\F_n$ be the partition of $f(A)$ into $8^n$ closed subintervals with same length.
    For any $y_0 \in [0, d_0] \setminus N$ by the definition of Lebesgue point we have
    \begin{equation}\label{property-of-mu-y}
        \mu_{y_0}(A) = \lim_{n \to \infty} \frac{\int_{I_n} \mu_y(A) dy}{\L^1(I_n)} = \lim_{n \to \infty} \frac{\mu(A \cap f^{-1}(I_n))}{\L^1(I_n)},
    \end{equation}
    where for each $n\in \N$ the interval $I_n \in \F_n$ is such that $y_0 \in I_n$.
    
    Since $A$ can be written as a disjoint union of $16^n$ rectangles of equal measure $\mu$,
    and the curvilinear horizontal stripe $A \cap f^{-1}(I_n)$ contains $2^n$ such rectangles (see Figure~\ref{pic3}), we have
    \begin{equation}
    	\mu(A \cap f^{-1}(I_n)) = 2^{n} 16^{-n} \mu(A).
    \end{equation}
    By construction $I_n \in \F_n$, so $\L^1(I_n) = 8^{-n} \L^1(f(A))$. Thus by \eqref{property-of-mu-y}
    for all $A \in \mathscr{R}_m$ we have
    \begin{equation}
    	\mu_{y_0}(A) = \frac{\mu(A)}{\L^1(f(A))} = \frac{16^{-m} \mu(C_0)}{8^{-m} d_0}.
    \end{equation}
    
    By construction the measure $\mu$ is concentrated on $C$ and therefore for a.e. $y$ the measure $\mu_y$ in concentrated on $C$.
    Let $x\in C$ be such that $y_0 := f(x) \in [0, d_0] \setminus N$.
    Then for any $m\in \N$ there exists $A_m \in \mathscr{R}_m$ such that $x\in A_m$ and therefore
    \begin{equation}
    	\mu_{y_0}(\{x\}) \le \mu_{y_0}(A_m) = O(2^{-m}), \quad m\to \infty.
    \end{equation}
    Hence $\mu_{y_0}(\{x\}) = 0$.
    Therefore $\mu_{y_0}$ has no atoms.
\end{proof}

\section{Acknowledgments}
The authors would like to thank three anonymous referees for their very fine comments.

\bibliographystyle{alpha}
\bibliography{references}

\end{document}

%% file: include/preambule.tex
\usepackage[colorlinks,citecolor=green!50!black,linkcolor=blue]{hyperref}
\usepackage[utf8]{inputenc}
\usepackage[T2A]{fontenc}
\usepackage[russian,english]{babel}
\usepackage[margin=2.54cm,includeheadfoot]{geometry}
\usepackage{amsfonts,amsmath} 
\usepackage{dsfont}
\usepackage{amssymb}
\usepackage{mathrsfs}
\usepackage{graphicx}
\usepackage{tikz}
\usetikzlibrary{intersections,calc}
\usepackage{subcaption}
\usepackage{enumitem}

\usepackage{tikz}
\usetikzlibrary{fadings}

\date{\today}

\let\d\relax
\newcommand{\d}{\partial}

\newcommand{\vi}{\boldsymbol{v}}

\newcommand{\R}{\mathbb{R}}

\newcommand{\N}{\mathbb{N}}
\newcommand{\F}{\mathscr{F}}
\newcommand{\fhi}{\varphi}
\newcommand{\1}{\mathds{1}}
\newcommand{\nablap}{\nabla^\perp}
\renewcommand{\L}{\mathscr{L}}
\renewcommand{\H}{\mathscr{H}}
\renewcommand{\div}{\operatorname{div}}
\renewcommand{\geq}{\geqslant}
\renewcommand{\leq}{\leqslant}

\DeclareMathOperator{\osc}{osc}
\DeclareMathOperator{\supp}{supp}
\newcommand{\loc}{\mathrm{loc}}

\newcommand{\qq}{\infty}

\newtheorem{thm}{Theorem}[section]
\newtheorem{lem}[thm]{Lemma}
\newtheorem{prop}[thm]{Proposition}
\newtheorem{defn}[thm]{Definition}

\definecolor{darkred}{rgb}{0.5,0,0}

\graphicspath{{images/}}

\usepackage{mathtools}
\mathtoolsset{showonlyrefs}

